\documentclass{amsart}

\usepackage{amsmath, amsthm}
\usepackage{amssymb, latexsym}
\usepackage{amsfonts}
\usepackage[dvipsnames]{xcolor}

\newtheorem{theorem}{Theorem}[section]

\newtheorem{proposition}[theorem]{Proposition}
\newtheorem{lemma}[theorem]{Lemma}
\theoremstyle{definition}
\newtheorem{remark}[theorem]{Remark}

\newtheorem{example}[theorem]{Example}

\newcommand{\FFF}{\mathbb F}
\newcommand{\CCC}{\mathbb C}

\newcommand{\LL}{{\mathcal L}}

\newcommand{\semi}{\rtimes}

\def\unA{{\mathbb A}{\rm ut}(\LL)}

\begin{document}
\date{}


\title{On 2-transitive sets of equiangular lines}
 
 \author{Ulrich Dempwolff and William  M. Kantor}



\begin{abstract}
All finite sets of equiangular lines spanning finite-dimensional unitary spaces  
are determined for which the action on the lines of the set-stabilizer in the unitary group is
 2-transitive with a regular normal subgroup.

 \end{abstract}

\maketitle

\section{Introduction}
\label{Introduction}

The purpose of this note is to verify the following result.

\begin{theorem}
\label{Main}

Let $\LL$ be a $2$-transitive set of equiangular lines
in the unitary space $V$ and that the automorphism group
of $\LL$ has a regular normal subgroup.
Let $|{\LL}|=n$,
 $\dim V=d$, and  $1< d < n-1$.
Then:

 \begin{itemize}

\item[\rm(i)] $n=4$ and $d=2$. 
\item[\rm(ii)] $n=64$ and $d=8$ or $=56$.  
\item[\rm(iii)]  $n=2^{2m}$ and $d=2^{m-1}(2^m-1)$ or $=2^{m-1}(2^m+1)$
 for $m\geq 2$.
\item[\rm(iv)]  $n=p^{2m}$ and $d=p^m(p^m-1)/2$ or 
$=p^m(p^m+1)/2$  for a prime
$p>2$ and $m\ge1$. 

 \end{itemize} 
 For each pair $(n,d)$ in  {\rm(i)--(iv)}  
 there is a unique such  set  $\LL$ up to equivalence.

\end{theorem}

A set  $\LL$ of \emph{equiangular lines}  in a 
unitary vector space $V$ is a set of   $1$-spaces 
that generates $V$ such that the
angle between any two members of $\LL$ is constant.
Two sets of lines are \emph{equivalent} if there
is a unitary transformation sending one set to the other.
The \emph{unitary automorphism group} $\unA$ of $\LL$
is the set  of unitary transformations sending $\LL$ to itself;
  the \emph{automorphism group} ${\rm Aut}\LL$
of $\LL$ is the group of permutations  of $\LL$ induced by $\unA$.  

We are assuming that ${\rm Aut}\,{\LL}$ is 
2-transitive. 
The case where ${\rm Aut}\,{\LL}$   has no regular normal subgroup occurs in \cite{IM}.  
The case $n=d^2$ is completely settled in \cite{Zhu}
producing (i), (ii) (and the case $n=3^2=d^2$ of (iv)),
while the corresponding question over the reals is implicitly dealt with in 
 in \cite{Ta2} (producing (iii)).
The assumption $1<d<n-1$ excludes degenerate examples (cf. \cite{IM}).

The proof of the theorem uses the classification of  the finite $2$-transitive
groups (a consequence of the classification of the finite simple groups),
together with mostly standard group theory and 
 representation theory. 
 We start with general observations concerning a 2-transitive line-set
 $\LL$ in a unitary space $V$.
In Section~\ref{Basic properties} we  show that 
$\unA = Z({\rm U}(V)) G$, where 
 $G$ is a finite group  $2$-transitive  on $\LL$; and then that
   $V$ is   an irreducible $G$-module.
 The stabilizer $H=G_\ell$ of   $\ell\in{\LL}$
has a linear character $\lambda$ such that, if  $W$  is  the module that  affords
the induced character $\lambda^G$, then 
  $W=V\oplus V'$ for a second irreducible $G$-module
$V'$ (Remark~\ref{RIrreducible}), 
which   explains why $2$-transitive line sets
{occur in pairs in the theorem}.  (See~\cite[p.~3]{IM}  for another 
explanation of this fact using Naimark complements.)
Then  we specialize to the case 
where  ${\rm Aut}\LL$ has a 2-transitive subgroup with a regular normal subgroup. 

Section~\ref{Group-theoretic background} contains  group-theoretic background and
Section~\ref{Examples}  describes the examples in  Theorem~\ref{Main}(iii) and (iv), while Section~\ref{proof}  contains the proof of  the theorem.
In the theorem  $\unA$ and ${\rm Aut}\,{\LL}$
are as follows:

\remark
\label{AutGroup}
For $\LL$ in  Theorem~\ref{Main}$,$
 $\unA =GZ$, $Z=Z({\rm U}(V))$ where $G=E\semi S$ with a $p$-group
 $E$  and  $H =G_\ell$  is  $Z(G)\times S$, $Z(G)=E\cap Z$.
In Section~\ref{proof}
 we prove that  
 the following hold for
the various cases in the theorem$:$

\begin{enumerate}
	\item[(i)]  $E=Q_8,$ 
	$|S|=3$ and $Z(G)=Z(E)$ has order $2;$

		\item[(ii)] $E$ is the central product of an extraspecial group of order
	$2^7$ with a cyclic group of order $4,$  
	$S\simeq {\rm G}_2(2)'\simeq {\rm PSU}(3,3)$
	and $Z(G)=Z(E)$ has order $4;$
		
	\item[(iii)]  $E$ elementary abelian of order $2^{2m+1},$
	 $S\simeq {\rm Sp}(2m,2)$ and $Z(G)= E\cap Z$ 
	 has order $2;$  and

	\item[(iv)]   
	$E$ is extraspecial of order $p^{2m+1}$ and exponent $p,\,$
	   $S\simeq {\rm Sp}(2m,p)$
	 and $Z(G)=Z(E)$ has order $p$.

\end{enumerate}

  
\section{Group theoretic background}
\label{Group-theoretic background}
 
Many facts of this section are basic and covered
in the books of Aschbacher~\cite{As},
 Huppert and Blackburn~\cite{HB}.
Our notation will follow the conventions of these references.
However we need the classification of the $2$-transitive
finite groups which is
is not elementary.
Such a group is either non-abelian quasisimple or it possesses a 
normal, regular subgroup. We are interested in the second case (the so-called
\emph{affine type}). For instance
Liebeck~\cite[Appendix~1]{Li} lists these groups:

\begin{lemma} 
\label{2tran}

Let $G$ be a finite $2$-transitive permutation group, $V\unlhd G$
 an
elementary abelian regular normal  subgroup 
of order $p^t $ for  a prime $p$. Identify $G$ with a group of affine transformations $x\mapsto  x^g+ c$ of $V=\FFF_p^t,$ where $g \in G_0$ and   $0, c\in  V$.
Then $G$ is a semidirect product $V\semi   G_0$
with $G_0\le {\rm GL}(V)$ ,and one of the following occurs:
\begin{enumerate}
\item  $G_0\le \Gamma {\rm L}(1, p^t).$ 
\item $G_0\unrhd {\rm SL}(s, q),$  $ q^s = p^t,$ $s> 2$.
\item $G_0\unrhd  {\rm Sp}(s, q),$ $ q^s = p^t.$   
\item $G_0\unrhd {\rm G}_2(q)', $ $q^6 = 2^t,$ where 
${\rm G}_2(q)<{\rm Sp}(6,q)\le {\rm Sp}(t,2)$.
\item $G_0$  is  $  A_6\simeq {\rm Sp}(4, 2)'$ or $A_7,$ $p^t=16$.
\item $G_0\unrhd {\rm SL}(2,3)$   with $t=2$ and
$p^t= 5^2,$ $ 7^2,$ $ 11^2 $ or $23^2$.
\item $G_0\unrhd   {\rm SL}(2,5)$   with $t=2$ and $p^t= 9^2,$ $
11^2, $ $19^2,$ $  29^2 $  or  $59^2  $.
\item $p^t = 3^4$   and $G_0 $    has a normal extraspecial subgroup $Q
=2^{1+4} $ such that 
$G_0= Q\semi  S  $ with $S\le {\rm O}^-(4,2)\simeq S_5$ and $|S|$ divisible by $5$.
\item $G_0'$  is  ${\rm SL}(2, 13),$  $p^t=3^6.$ 
\end{enumerate} 

\end{lemma}

\subsection{Some indecomposable modules}

Let $U$ be an elementary abelian $p$-group (written additively) and
$S\leq {\rm Aut}(U)$, i.e. we  consider $U$ as a faithful
$\FFF_pS$-module.
We say that that $U$ is \emph{indecomposable} if
$U$ is not the direct sum of two proper $S$-submodules.
We are interested in modules
with the following properties:

 
\medskip

\noindent
{\bf Hypothesis} (I)
\begin{enumerate}
	\item[({\rm I})] $U$ has a trivial $S$-submodule $0\neq U_0$ 
	 and $S$ acts transitively on the non-trivial
	elements of $V=U/U_0$. Moreover:
	
\begin{enumerate}
	\item The proper submodules of $U$ are contained $U_0$.

	\item $\dim U_0=1$.
\end{enumerate}
	
\end{enumerate}

The possible pairs $(S,V)$ are listed in 
Lemma~\ref{2tran} ($S$ in the role of $G_0$).
The module $U$ is an indecomposable module
which extends the trivial module by $V$ in case (I.a).


\begin{lemma}
\label{NonVan}

Let $U$ be an indecomposable $\FFF_pS$-module satisfying (I).
 Then $p=2$ and the following hold:

\begin{enumerate}

	\item[(a)] $\dim V=2m$, $m>1$,
	$S\simeq {\rm Sp}(2a,2^b)'$,  $m=ab$, or
  $\simeq {\rm G}_2(2^b)'$, $m=3b$.
  
  \item[(b)] $\dim V=3$, $S={\rm SL}(3,2)$.
  
  \item[(c)] The module $U$ exists in cases (a) and (b) and  is
unique up to equivalence.
 
  \item[(d)] Let $S\simeq {\rm Sp}(2a,2^b)'$,  $m=ab$, or
  $\simeq {\rm G}_2(2^b)'$, $m=3b$. Then $S$ has an
  embedding into a group $S^\star\simeq {\rm Sp}(2m,2)$ and
  $U$ is the restriction of
  the unique $\FFF_2S^\star$ (satisfying (I)) to $S$.

\end{enumerate}
 
\end{lemma}

Before we start the proof we recall a few basic facts about
group representations and cohomology.
Let $G$ be a finite group, $V$ be an $n$-dimensional $FG$-module associated
with the matrix representation $D:G\to {\rm GL}(n,F)$.
Define $D^*: G\to {\rm GL}(n,F)$ by $D^*(g):= D(g^{-1})^t$.
With respect to $D^*$ the space $V$ becomes a $G$-module,
\emph{the dual module $V^*$ of $V$}.

We describe the connection of the existence of 
indecomposable modules with cohomology of degree $1$ and
follow Aschbacher~\cite[Sec. 17]{As}.
Let $G$ be a finite group, $V$ a finite dimensional, faithful
$\FFF_pG$-module.
A mapping $\delta : G\to V$ is called a \emph{derivation or $1$-cocycle}
if $\delta (xy)=\delta (x)y+\delta (y)$ for all $x,y\in G$.
If $v\in V$, then $\delta_v$ defined by $\delta_v(x)= v- vx$
is a derivation too. Such derivations
are called \emph{inner derivations or $1$-coboundaries}.
The set ${\rm Z}^1(G,V)$ of derivations
and the set ${\rm B}^1(G,V)$ of inner derivations become
 elementary abelian $p$-groups
with respect to pointwise addition.
The factor group
$$
{\rm H}^1(G,V)={\rm Z}^1(G,V)/{\rm B}^1(G,V)
$$
is the \emph{first cohomology group of $G$ with respect to $V$}.

Suppose, $V$ is a simple $G$-module. By Schur's Lemma
$K={\rm End}_{\FFF_pG}(V)$ is a finite field, say $\simeq \FFF_{p^e}$,
$e\mid \dim V$. For $\kappa \in K$, $\delta$ a derivation define
$\delta\kappa :G\to V$ by $\delta\kappa (x)= \delta (x)\kappa$.
Then $\delta\kappa$ is a derivation
and $\delta_v\kappa=\delta_{v\kappa}$. So ${\rm Z}^1(G,V)$,
${\rm B}^1(G,V)$ and ${\rm H}^1(G,V)$ become $K$-spaces.
Usually ${\rm H}^1(G,V)$ is described by $\dim _K{\rm H}^1(G,V)$, i.e.
${\rm H}^1(G,V)$ has dimension $[K:\FFF_p] \cdot \dim _K{\rm H}^1(G,V)$
over $\FFF_p$.

By~\cite[(17.12)]{As} we have

\begin{enumerate}
	\item[(i)] There exists a $\FFF_pS$-module 
	with property (I.a) iff ${\rm H}^1(S,V^*)\neq 0$.
  
  \item[(ii)]  Every $\FFF_pS$-module 
	with property (I.a) is a quotient
  of a uniquely
  determined $\FFF_pS$-module
  $W$ with property (I.a) such that 
   $\dim C_{W}(G) =\dim {\rm H}^1(S,V^*)$. 

\end{enumerate}
If $V^*$ is simple then the module $W$ in (ii) is
even a $K$-module. So if $U$ satisfies (I) then
there exists a hyperplane $W_0$ of $C_{W}(G)$
such that $U\simeq W/W_0$. If $\dim _K{\rm H}^1(S,V^*)=1$,
then the multiplicative group of $K$ acts transitively
on the hyperplanes of $C_{W}(G)$, i.e. any two
modules satisfying (I) are isomorphic.


\begin{proof}
Assume $p>2$. Inspecting the cases of
 of Lemma~\ref{2tran}
we see that $Z(S)$ contains an involution $z$
which induces the scalar $-1$ on $V$. If $U_-$ denotes
the eigenspace of $z$ for the eigenvalue $-1$
then $U=U_0\oplus U_-$ is a $G$-decomposition, a contradiction.

So $p=2$ and we have to consider cases (2)-(5) of Lemma~\ref{2tran}
for $S$. Assume $\dim _{\FFF_p} V= p^t$.
In cases (2) - (4) $S\simeq {\rm SL}(a,2^b)$, $ab=t$,
${\rm Sp}(2a,2^b)'$, $2ab=t$, and ${\rm G}_2(2^b)'$, $3b=t$
and $V$ is the defining
$\FFF_{2^b}S$-module. 
In case (2) we get assertion (b) by~\cite{JP}.
In case (3) and (4) ${\rm H}^1(S,V^*)$ has dimension $1$ over
$\FFF_{2^b}$ by~\cite{JP}. 
It follows 
that a module with property (I) exists and is unique up
to isomorphism.
Assertions (a) and (c) hold.
In (5) with $S\simeq {\rm A}_7$ we have  ${\rm H}^1(S, W^*)= 0 $: 
Let $U$ be a $5$-dimensional $\FFF_2S$-module such that $\dim C_U(S)=1$
and $U/C_U(S)$ is simple. There are $16$ hyperplanes in $U$
that intersect $C_U(S)$ trivially. A permutation representation
of $S$ of degree $\leq 16$ has degree $1$,$7$ or $15$.
Hence $C_U(S)$ has an $S$-invariant complement in $U$.
 
 To (d): Recall that 
 $S\simeq {\rm Sp}(2a,2^b)'$~\cite[Hilfssatz 1]{Hu}, $ab=m$, and 
$\simeq {\rm G}_2(2^b)'$, $3b=m$~\cite[p.~513]{Li}, are
 subgroups of $S^\star = {\rm Sp}(2m,2)\simeq {\rm O}(2m+1,2)$ 
 and that the indecomposable
$S^\star$-module $U$  is the ${\rm O}(2m+1,2)$-module
\cite[pp.~55,~143]{Ta1}.
As $S$ acts transitively on
$V\simeq U/U_0$
we see that $U$ is indecomposable as an $S$-module.
\end{proof}



\subsection{On representations of extraspecial groups}
\label{SExtraspecial}

We consider:

\medskip

{\bf Hypothesis} (E) \ $p$ is a prime and $m\geq 1$ an integer.
If $p>2$ then $E$ is an extraspecial group of order $p^{1+2m}$
and exponent $p$ and if $p=2$ then $E$ is the central product
of an extraspecial  of order $2^{1+2m}$ with a cyclic group of
order $4$.

\medskip
Assume Hypothesis (E) and let
 $A=\{ \alpha \in {\rm Aut} (E) \mid \alpha_{Z(E)}=1_{Z(E)} \}$
be the centralizer of $Z(E)$ in the automorphism group.
Then (see~\cite{Gri},~\cite{Wi})
\begin{equation}
\label{Eaut}
A/{\rm Inn}(E) \simeq {\rm Sp}(2m,p).
\end{equation}
Denote by $\zeta_k=\exp( 2\pi i/k)$ a $k$-th primitive
root of unity. Assertions (a) and (b)
of  the next Lemma  are~\cite[(34.9)]{As},~\cite[Satz V.15.14]{HB}
whereas the last assertion follows from~\cite[Thm. 1]{Wi}. 

\begin{lemma}
\label{Ex}
Assume Hypothesis (E) and let
$U$ be a $p^m$-dimensional complex space.
 Set $Z(E)=\langle z\rangle$.

\begin{enumerate}
	\item[(a)] If $p=2$ there exists precisely two faithful, irreducible
	representations $D_j: E\to {\rm GL}(U)$, $j=1,3$.
	One has $D_j(z)=\zeta_4^j \cdot 1_U$.
	Every faithful, irreducible representation of $E$ is of this form.
	
	\item[(b)] If $p>2$ there exists precisely $p-1$ faithful, irreducible
	representations $D_j: E\to {\rm GL}(U)$, $1\leq j\leq p-1$.
	One has $D_j(z)=\zeta_p^j \cdot 1_U$.
	Every faithful, irreducible representation of $E$ is of this form.
 
\end{enumerate}
For each $j$ there is an automorphism $\gamma_j$ of $E$ such that $D_j$ can be defined by
$D_j(e)=  D_1(e^{\gamma_j})$ for all $e\in E,$ so $D_j(E)=D_1(E)$.
\end{lemma}



\subsection{Basic properties of $2$-transitive line sets}
\label{Basic properties}

In this subsection $\LL$ denotes a 
$2$-transitive set of $n$ equiangular lines
in a a unitary space
$V$ of dimension $d<n$. 
Let $K$ be the kernel of the permutation action of 
$\unA$ on $\LL,$ which  clearly contains $Z :=Z({\rm U}(V))$.

\begin{lemma}
\label{Kernel}

$K= Z$.

\end{lemma}

\begin{proof}
Let $g\in K$.
 Let $m$ be the minimal number of nonzero  $a_i$
  in a  dependency relation $\sum_i a_iv_i=0$,  $\langle v_i\rangle \in \LL$.
   Apply $g$ to obtain another  dependency relation   
   $\sum_i k_ia_iv_i=0$ with the same
   $m$ nonzero $k_ia_i$;    these    relations must be multiples of one another by minimality.  
Thus, restricting to nonzero  $a_i$ produces
  constant $k_i$.
  
Any two different members $\langle v_i\rangle, \langle v_j\rangle$ 
  of  $\LL$   occur with nonzero coefficients in such a relation.
Then $g$ acts on all  members of $\LL$ with the same~scalar, and so
 is a scalar transformation  since $\LL$ spans $V$.
\end{proof}

\begin{lemma}
\label{Finite}
 
There is a finite group $G$ such that $\unA =GZ$.
\end{lemma}

\begin{proof}
By \cite[(33.9)]{As}, $D=\unA'$ is finite.  Let $D\leq G$ be a finite group such that $GZ/Z$ has 

maximal order in the finite group ${\rm Aut}\LL= \unA /Z$.   
If $GZ<\unA$ let 
$h\in \unA -GZ$.
Then $h^m\in Z$ for some integer $m$, so there is $z\in Z$ such that
$h^m=z^{-m}$.  Since $[G,hz]\subseteq D\le G$
we get $|\langle G, hz\rangle|<\infty$
and $GZ/Z< \langle G, h\rangle Z/Z= \langle G, hz\rangle/Z$,
a contradiction. 
\end{proof}


\begin{proposition}
\label{Irreducible}
Let $H=G_\ell$, $\ell\in {\LL}$
be the stabilizer of a line.
Let $\lambda$ be the linear character of $H$ afforded by 
$\ell$.
Then:
\begin{enumerate}
	\item[(a)] $V$ is simple and
	a constituent of the module $W$ which affords $\lambda^G$.
	\item[(b)] $W=V\oplus V'$ with a simple module $V'$
	inequivalent to $V$.
	\item[(c)]  $V$ and $V'$ as $H$-modules afford $\lambda$ with
	multiplicity $1$.
\end{enumerate}

\end{proposition}

\begin{proof} 
By 2-transitivity,  $G=H\cup HtH$ for $t\in G-H$.
Assume that $V=V_1\oplus\cdots\oplus V_r$ for simple $G$-modules $V_i$.
Let $\chi_i$ be the character of $V_i$.

Let $\ell =\langle v\rangle$.
If $v=v_1+\cdots +v_r$ with $v_i\in V_i$ then each $v_i\ne0$ since 
$\langle \LL \rangle =V$.  
As $\lambda(h)v=\lambda(h)v_1+\cdots +\lambda(h)v_r$ 
for $h\in H$, $\lambda$ is a constituent of
$(\chi_i)_H$.
By Frobenius Reciprocity, each $\chi_i$ is a constituent of
$\lambda^G$.

We claim that \em $\lambda^G=\psi_1+\psi_2$ for distinct irreducible characters 
$\psi_i$ of $G$.  \rm
 For, by  Mackey's Theorem \cite[p.~557]{HB},
$(\lambda^G)_H=((\lambda^{1^{-1}}) _{H\cap H^1})^H   + 
((\lambda^{t^{-1}} )_{H\cap H^t}) ^H.
$
By Frobenius Reciprocity,  $(\lambda^G,  \lambda^G)=
(\lambda, (\lambda^G)_H) = 1+(\lambda,((\lambda^{t^{-1}} )_{H\cap H^t}) ^H)$ 
and 
$(\lambda,((\lambda^{t^{-1}} )_{H\cap H^t}) ^H)=
(\lambda _{H\cap H^t},(\lambda^{t^{-1}} )_{H\cap H^t})$.
Hence $(\lambda^G,  \lambda^G)=1$
or 2.  If $\lambda^G$
is irreducible then each $\chi_i=\lambda^G$,  so 
$d=r\lambda^G(1)=r|\LL|\ge n$.   This contradiction proves the claim.
By Frobenius Reciprocity, $(\lambda,(\psi_i)_H)=1$ for $i=1,2$.
Then (a) - (c) follow if $r=1$.

So assume $r>1$. Each $\chi_i$ is in
 $\{ \psi_1,\psi_2 \}.$   
 If $ \{ \chi_1,\chi_2 \}=\{ \psi_1,\psi_2 \}$ then $d\ge \chi_1(1)+\chi_2(1)=
\lambda^G(1)=|\LL|$, which is not the case.

Since $\psi_1\ne\psi_2$ we are left with the possibility 
$\chi_1=\chi_2\in \{ \psi_1,\psi_2 \}$, say $\chi_i=\psi_1$.  
Let $\phi\colon V_1\to V_2$ be a $G$-isomorphism.  
Since  $\lambda$ has multiplicity $1$ in $\psi_1$ the morphism
 $\phi$ sends~the unique submodule of $(V_1)_H$  affording  $\lambda$
to the unique submodule of  $(V_2)_H$  affording  $\lambda$.
Thus $v_1\phi=av_2$ with $a\in \CCC^*$.
Then  
$\langle v_1g+v_2g\mid g\in G \rangle
=\langle v_1g+a^{-1}v_1\phi  g\mid g\in G \rangle =V_1(1+ a^{-1}\phi )$,
showing
$\langle {\LL} \rangle \subseteq V_1(1+ a^{-1}\phi )\oplus V_3\oplus \cdots
\oplus V_r$. This
contradicts the fact that  $\LL$ spans~$V$.
\end{proof}

\begin{remark}
\label{RIrreducible}
 $\lambda$ is a
\emph{non-trivial} character for
$1< d< n-1$ ($((1_H)^G,1_G)=1$ by
Frobenius reciprocity) and clearly $V'$ contains a $2$-transitive line set too.

\end{remark}


\section{Examples of $2$-transitive  line sets}
\label{Examples}

In this section we describe the examples of Theorem~\ref{Main}.
See~\cite{Ho} and~\cite{Zhu} for Theorem~\ref{Main}(i,ii).

\begin{example}
\label{OrthOddDim}
{\bf  for Theorem~\ref{Main}(iii)}
Let $m>1$.
Let $E  = \FFF_2^{2m+1}$  be an ${\rm O}(2m+1,2)$-space  with radical $R$  \cite[pp.~55,~143]{Ta1}.
Then $S:=   {{\rm O}(2m+1,2)}\simeq {\rm Sp}(2m,2)={\rm Sp}(E/R)$ is transitive on the 
  $d:= 2^{m-1}{(2^m-1)}$ hyperplanes of $E$ of type ${\rm O}^-(2m,2)$
and on the $2^{m-1}(2^m+1)$ hyperplanes  of type ${\rm O}^+(2m,2)$ \cite[p.~139]{Ta1}.
 Label the standard basis elements of $V= \CCC ^d$ as $v_M$ with $M$ ranging over  the first of these sets of hyperplanes. 
 Let $S$ act on this basis as it does on these hyperplanes.
 This action is 2-transitive (as observed implicitly for line-sets in \cite{Ta2} and
 first observed in \cite{Di}), so the only irreducible
 $S$-submodules of $V$ are $\langle \bar v\rangle $  and $\bar v^\perp$, 
 where   $\bar v:=\sum_Mv_M$.

Each such $M$   is the kernel of  a  unique character $\lambda_M\colon E\to \{\pm1\}$.   Let $e\in E$ act on $V$ by
   $v_Me:= \lambda_M(e)v_M$ for each basis vector $v_M$.  
 If  $1\ne r\in R $ then 
 $\lambda_M(r)=-1$  since  $r\notin M$, so $r$ acts  as $-1$ on $V$. 
 If  $e\in E$ and $h\in S$ then  
 \vspace{2pt}
  $  (\bar v e )h=  \bar v h\cdot h^{-1}eh
  = \bar v e^h$,  so $S$ acts on  
  $\langle \bar v\rangle E$, a  set  of 1-spaces of $V$.
 Since $S$ is irreducible on $\bar v^\perp$, the set 
 $\langle \bar v\rangle E=\langle \bar v\rangle {ES}$ spans $V$
 and  $\langle \bar v\rangle$ is the only 1-space fixed by $S$.
In particular $\langle \bar v\rangle$ affords the unique
involutory linear character $\lambda$ of $H=R\times S$ whose kernel is $S$.
Clearly  $(E/R)\semi  S$ acts 2-transitively on the $n=2^{2m}$  cosets of $S$. 
These are the $d$-dimensional examples in Theorem~\ref{Main}(iii). The $2^{m-1}{(2^m+1)}$
 hyperplanes of type ${\rm O}^+(2m,2)$ produce
similarly the   $(n-d)$-dimensional examples. 
\end{example}

\begin{example}
\label{Weil}
{\bf for Theorem~\ref{Main}(iv)}
Let $p>2$ be a prime, $m$ a positive integer
and  $E$ be an extraspecial group of  order $p^{1+2m}$
and exponent $p$.
Using Lemma~\ref{Ex} we consider $E$ as a subgroup
of ${\rm U}(W)$, $W$ a unitary space of dimension $p^{m}$.
By~\cite{BRW} the normalizer of $E$ in ${\rm U}(W)$
contains a subgroup $G=E\semi S$, $G/E\simeq {\rm Sp}(2m,p)$
inducing ${\rm Sp}(2m,p)$ on $E/Z(E)$, with $ES$  acting
  2-transitively on the  $n=p^{2m}$    cosets of $S$. 
Moreover,  $Z(S)=\langle z\rangle$ has order 2, and
 $W=W_+\perp W_-$ for the  eigenspaces  
$W_+$ and $ W_-$ of  $z$ 
of respective dimensions  $(p^m-1)/2$ and $(p^m+1)/2$;  
these are irreducible $S$-modules   (\!{\em Weil modules})  \cite{BRW,Ge}.

Let $U$ be
one of these eigenspaces
say of dimension $d$. As $G/E\simeq S$ we can consider $U$
as a $G$-module. Define
 $V:= W\otimes U^* \subset W\otimes W^*$
($U^*$ dual to $U$).
If $\chi $ is the character of $S$ on $U$ then $\chi \bar\chi $ 
is the character of $S$ on $U\otimes U^*$.   Trivially, 
$(\chi\, \bar\chi,1_S)=(\chi ,\chi)=1$,  so there is a unique 1-space
$\langle v_0\rangle $ in $U\otimes U^*$ (and hence in $V$) fixed pointwise by $S$
(and it is the only   1-space fixed by  the simple group $S$).
In particular $\langle  v_0\rangle$ affords a non-trivial
 linear character $\lambda$ of $H=Z(E)\times S$ with kernel $S$.
Since $E$ is irreducible on $W$ while $S$ is irreducible on $U^*$,
the set $\langle v_0\rangle ^{ES}$ spans $V$.
These are the examples in Theorem~\ref{Main}(iv).

\end{example}


\begin{lemma}
\label{Unique}

Let  $\LL$ be a line set of size $n=p^{2m}$, $p$ a prime, in a unitary
space $V$ of dimension $1<d\leq n/2$.
Let $G\leq \unA$ induce a $2$-transitive action
on $\LL$. Assume further  $G\simeq
E\semi S$, $S\simeq {\rm Sp}(2m,p)$ with $E$ elementary abelian of order $2^{1+2m}$ 
for $p=2$ and $E$ is extraspecial of order  $ p^{1+2m}$ for $p>2$.
Then $\LL$ is equivalent to a line set of Example~\ref{OrthOddDim} 
or~\ref{Weil}.

\end{lemma}

\begin{proof}
For $i=1,2$ let $\LL_i\subseteq V_i$ be line sets in unitary spaces,
$1< \dim V_i\leq n/2$.
Let $G\simeq G_i=E_i\semi S_i \leq {\rm U}(V_i)$ be a finite group
acting $2$-transitive on $\LL_i$.
Let $\ell_i\in \LL_i$ and $H_i=(G_i)_{\ell_i}$.

\medskip

{\sc Claim:} $\LL_1$ is equivalent to $\LL_2$, in
particular the assertion of the Lemma holds.

\medskip

By Proposition~\ref{Irreducible} and Remark~\ref{RIrreducible} the representation $\lambda_i$
of $H_i$ on $\ell_i$ is a non-trivial linear character of $H_i$.
We have $H_i=Z_i\times S_i$, $Z_i=Z(G_i)$, $S_i\simeq {\rm Sp}(2m,p)$.
Let $\alpha : G_1\to G_2$ be an isomorphism.

\medskip 

{\sc Case $p>2$.} The group
 $S_i$ is a representative of the unique class
 of complements
of $E_i$ in $G_i$ (note that $S=C_{G}(Z(S))$
and $Z(S)$
is a Sylow 2-subgroup of $E\semi Z(S))\unlhd G$).
So we can assume $H_2=H_1\alpha$, $S_2=S_1\alpha$.
We also know $S_i=\ker \lambda _i$ by Lemma~\ref{Cases} below.
By Lemma~\ref{Ex} there exists an automorphism $\gamma$
of $G_1$ such that $\lambda _1(z)=\lambda_2(z\gamma\circ \alpha)$
for $z\in Z$. So replacing if necessary $\alpha$ by $\gamma\circ \alpha$
we may assume that $\lambda _1(z)=\lambda_2(z\alpha)$ holds.
Define a representation
$D: G_1\to {\rm GL}(V_2)$ by
$$
v_2D(g)=v_2 (g\alpha), \quad v_2\in V_2,\quad g\in G_1.
$$
Let $W$ be the module associated with the induced character
$\lambda_1^{G_1}$. By Proposition~\ref{Irreducible} both $G_1$-modules
are isomorphic to the same irreducible submodule of
$W$ (use the restriction on the dimension
and that one of the line sets may come
from the example), i.e. $V_1\simeq V_2$. Hence there exists a $G_1$-morphism
$\phi: V_1\to V_2$ with $\ell_1\pi= \ell_2$ 
($\lambda_1$ has multiplicity $1$ in $V_1$ and $V_2$).
The claim holds for $p>2$.

\medskip 

{\sc Case $p=2$.}  Assume first $m>2$.
$S_2$ and $S_1\alpha$ are 
 complements of $E_2$ in $G_2$.
 By~\cite[(17.7)]{As} there exist $\beta \in {\rm Aut}(G_2)$ with
 $S_2=(S_1\alpha)\beta$.
So replacing $\alpha$ if necessary by $\alpha \circ \beta$
we can assume $H_1\alpha =H_2$ and $S_1\alpha =S_2$.
Note that $H$ has precisely one non-trivial
linear character.
Now arguing as in the case $p>2$
we see that $\LL_1$ and $\LL_2$ are equivalent.

In case $m=2$ replace $S$ by $S'$ (and $S_i$ by $S_i'$). Then
the argument from  case $m>2$ carries over
and shows the equivalence of $\LL_1$ and $\LL_2$.
\end{proof}

\begin{remark}
\label{RUnique}
(a)  If we replace  in Lemma~\ref{Unique} 
the condition $1<d\leq n/2$ by $n/2\leq d< n-1$
we get (by symmetry) again the uniqueness assertion
of the Lemma.
(b)  For \emph{any} linear character $\lambda$ of $H$ with kernel $S$
the constituents of the module associated with
$\lambda^G$ contain $G$-admissible line sets.

\end{remark}


\section{Proof of Theorem~\ref{Main} and automorphism groups}
\label{proof}

In this section $p$ is a prime and
${\LL}$ denotes a set of $n=p^t$ equiangular lines
in a an unitary space
$V$ of dimension $d$ such that $G\leq \unA$
is a finite group with a $2$-transitive action on $\LL$.
We also assume $n\neq 4$ as for $n=4$
\cite{Zhu} implies Assertion (i) of Theorem~\ref{Main}.
Using the results of Subsection~\ref{Basic properties}
 we  assume that $1<d < n-1$, $V$ is a simple
$G$-module and $G/Z$, $Z=Z(G)$ has a regular normal
subgroup.
 It is sufficient to assume that no proper subgroup of
$G/Z$ has a $2$-transitive action on ${\LL}$
and that no subgroup of $\unA$, 
which covers the quotient
$GZ/Z$, has  order $<|G|$.
We set
 $H=G_{\ell}$, $\ell\in {\LL}$.
Then the character/representation $\lambda : H\to {\rm U}(\ell)$ 
 of $H$ on $\ell$ is non-trivial by Remark~\ref{RIrreducible}.
Observe that there is some flexibility in the choice of
$G$: generators of $G$ can be adjusted by scalars. We show that
$G$ can be chosen such that $G\leq \tilde G$
where $\tilde G$ is a group which is used to construct
a line set in Examples~\ref{OrthOddDim} and~\ref{Weil}.


\begin{lemma}
\label{Cases}

We may assume
 $G=E\semi S$,
 $H=Z\times S$, where
 $S$ is the kernel of the action of $H$ on $\ell$. Moreover, $Z\leq E$
 and:

\begin{enumerate}
	\item[(a)] 
	$p=2$, $E$ is an elementary abelian $2$-group
 and $E$ as an $S$-module satisfies Hypothesis (I).
	
	\item[(b)]  $t=2m$,
	  $E$ satisfies
	Hypothesis (E) and $E/Z(E)$ is a simple $S$-module.
	
\end{enumerate}

\end{lemma}

\begin{proof}
Let $M$ be the pre-image of the regular, normal subgroup of $G/Z$.
Since $M/Z$ is abelian we have $M=E\times Z_{p'}$ with a Sylow $p$-subgroup $E$
of $M$ and $Z_{p'}$ is the largest subgroup of $Z$
with an order coprime to $p$. 
Let $L$ be the kernel of $\lambda$.

We may assume $E=M$, $Z\leq E$ and that
 $S=L$ is a complement of $Z$ in $H$:
Clearly, $Z\leq H$ and $L\cap Z=1$.
As $H/L$ is cyclic we can choose $c\in H$ such
that $H=\langle c,L\rangle$. Pick $\omega$ in $\CCC$ of norm $1$
such that $S=\langle \omega c,L\rangle$ has a trivial action
on $\ell$. Clearly, $S$ is finite and 
$\widetilde{G}=ES$ induces a $2$-transitive action on  $\LL$.
Moreover $S\cap E\leq S\cap (\widetilde{G}_\ell \cap E)
\leq S\cap Z({\rm U}(V)) =1$.
So we may assume $G=\widetilde{G}$ and $H=(E\cap Z)\times S$.
In particular $Z\leq E$ as $Z\leq H$.

Assume first that $E$ is abelian. 
Set $\Omega =\langle e\in E \mid |e|=p\rangle$. 
This group is a characteristic elementary abelian
subgroup of $E$.
If $\Omega \leq Z$ then $E$ is cyclic,
and $S\neq 1$ is a $p'$-group (isomorphic  isomorphic
to a subgroup of ${\rm Aut}(E)$
of order $p-1$). 
By Remark 2.7   $Z\neq 1$.
This contradicts~\cite[(23.3)]{As} (Automorphism groups of cyclic groups).

So $E = \Omega (Z\cap E)$  and
by the minimal choice of $G$ we obtain $E=\Omega$.
If $Z$ has a $S$-invariant complement $E_0$
in $E$ then by induction $G=E_0S$ contradicting $Z\neq 1$.
So $1< Z <E$ is the unique composition series of $E$ as an $S$-module
and 
Assertion (a) follows.

Assume now that $E$ is non-abelian.
If $N$ would be a characteristic, normal, abelian subgroup
of $E$ of rank $\geq 2$, then 
 $1< NZ/Z \leq E/Z$ is an $S$-invariant series.
By our minimal choice $N=E$, which is absurd.
So $E$ is of symplectic type and therefore by~\cite[(23.9), p. 109]{As}
$E=C\circ E_1$ where $E$ is extraspecial or $=1$ and
$C$ is cyclic or $p=2$ or $C$ is 
a generalized quaternion group, a dihedral
group or a semidihedral group of order $\geq 16$.

Suppose $p>2$. By~\cite[(23.11)]{As} is extraspecial
$E$ of exponent $p$.
Assertion (b) follows for $p>2$.

Suppose finally $p=2$. A standard reduction 
(see for instance~\cite[Lem. 5.12]{Th}) shows that $E$ contains
a characteristic subgroup $F$ such that
 $F$ is extraspecial of order $2^{1+2m}$ or satisfies hypothesis (E).
 By our choice of $G$ we have $E=F$ as $t=2m>2$.
If $E$ is extraspecial then
$S$ can not act transitively
on the non-trivial elements of $E/Z(E)$ as there are
cosets of elements of order $4$ as well as cosets of elements of order
$2$. Assertion (b) holds for $p=2$ too.
\end{proof}


By Lemma~\ref{Cases} we distinguish the cases  $E$ abelian ($p=2$),
 $E$ non-abelian, $p>2$ and  $E$ non-abelian, $p=2$. Then
 Lemmas~\ref{AbelianOdd} and~\ref{Even}  complete the
 proof of Theorem~\ref{Main}.
The proof of Lemma~\ref{AbelianOdd} is very similar to the proof
of Lemma~\ref{Unique}.

\begin{lemma}
\label{AbelianOdd}
\begin{enumerate}
	\item[(a)] Let $E$ be abelian. Then (iii) of Theorem~\ref{Main} holds.
	\item[(b)] Let $E$ be non-abelian and $p>2$. Then (iv) of Theorem~\ref{Main} holds.
\end{enumerate}

\end{lemma}

\begin{proof}
We first eliminate case (b) of Lemma~\ref{NonVan}.
Let $G=E\semi S$, $S\simeq {SL}(3,2)$, $Z=C_E(S)$ and $E/Z$ is the natural $S$-module.
A simple $E$-module in $V$ affords a non-trivial character $\chi$
of $E$ and its kernel $E_ \chi$ is a hyperplane intersecting $Z$ trivially.
There are precisely $8$ such hyperplanes.
The group $S$ acts transitively on these hyperplanes (otherwise ---
as the smallest degree of a non-trivial permutation representation of
$S$ is $7$ --- $S$ would fix one of these hyperplanes
and $E$ would not be an indecomposable $S$-module).
Hence $\dim V\geq 8 =n$, a contradiction.

There exist an embedding $\iota : G\to \tilde G$, 
$\tilde G= \tilde E\semi \tilde S$, $\tilde S\simeq {\rm Sp}(2m,p)$
with $\tilde E=  E\iota$, $S\iota \leq \tilde S$:
This follows from (d) of Lemma~\ref{NonVan} if $p=2$
and for $p>2$ this is clear by Equation~(\ref{Eaut}).
The linear character $\tilde \lambda$ of $H\iota$
defined by  
\begin{equation}
\label{Elam}
\tilde \lambda (h\iota )=\lambda( h ),\quad h\in H
\end{equation}
has a unique extension to $\tilde H= Z\iota \times \tilde S$
such that $\ker \tilde \lambda = \tilde S$.
Let $\tilde W$ be the module associated with the induced
character $(\tilde \lambda )^{\tilde G}$.
By 
Lemma~\ref{Unique} and Remark~\ref{RUnique} (b) we have
a decomposition into simple $\tilde G$-modules
$\tilde W= \tilde V\oplus \tilde V'$ and both modules
contain $\tilde G$-invariant line sets.
We turn $\tilde W$ into a $G$-module by 
$$
\tilde w\cdot g= \tilde w (g\iota),\quad \tilde w \in \tilde W,
\quad g\in G.
$$
By Mackey's Theorem~\cite[Satz V.16.9 (b)]{HB} and Equation (\ref{Elam})
$$
((\tilde \lambda)^{\tilde G})_G =
((\tilde \lambda)_{{\tilde H}\cap G\iota})^G
=(\lambda_{H})^G.
$$
So $\tilde W$ as a $G$-module affords $\lambda^G$.
Then by Proposition~\ref{Irreducible}
$V$ is isomorphic to $\tilde V$ or $\tilde V'$. Say $V\simeq \tilde V$.
Then (see the proof
of Lemma~\ref{Unique}) an isomorphism $\phi : V\to \tilde V$ 
maps the line set
$\LL$ onto a $\tilde G$-invariant line set in $\tilde V$.
\end{proof}



\begin{lemma}
\label{Even}

Let  $E$ be non-abelian and $p=2$. Then (i) or
(ii) of Theorem~\ref{Main} hold.

\end{lemma}

\begin{proof} 
By  Proposition~\ref{Irreducible}  we may assume 
$d=\dim V \leq n/2= 2^{2m-1}$.
 As $E$ satisfies Hypothesis (E) $S$ is isomorphic to a subgroup 
of ${\rm Sp}(2m,2)$  (see Equation~(\ref{Eaut})). 
By Lemma~\ref{2tran}
and by the minimal choice of $G$
we have $H/Z(H)\simeq {\rm SL}(2,2^m)$ or $\simeq {\rm G}_2(2^b)'$
and $b=m/3$.
Let $V=V_1\oplus \cdots \oplus V_\ell$ a decomposition
into irreducible $E$-modules, in particular
$d=2^m\ell$. Clearly, all $V_i$ are faithful
$E$-modules. Then a generator of $Z$  induces the same scalar
on each $V_i$ as the eigenspaces of this generator are $G$-invariant.
 Lemma~\ref{Ex} shows that 
all $V_i$'s are pairwise isomorphic.
If $\ell =1$, then $n=2^{2m} =d^2$ 
and an application of the main result of~\cite{Zhu} proves
 the assertion of the Lemma.

So assume $\ell >1$. 
Denote by $D$ the representation of $G$ afforded by $V$
and apply~\cite[Satz V.17.5]{HB}. Then
$D(g)=P_1(g)\otimes P_2(G)$ where the $P_i$'s are irreducible projective
 representations of $G$ and $P_2$ is also a projective representation of 
$S\simeq G/E$ of degree $\ell$. Denote by $m_S$
the minimal degree of a non-trivial projective representation
of $S$. By~\cite[Satz V.24.3]{HB} $m_S$ is the minimal degree
of a non-trivial, irreducible representation of the 
universal covering group
of $S$.
We have $m_S=2^m-1$ for 
$S\simeq {\rm SL}(2,2^m)$, $m>3$~\cite[Tab. 3]{TZ},~\cite{LS},
$m_S=2^m-2^b$ for $S\simeq {\rm G}_2(2^b)$, $m=3b$, 
$b\neq 2$~\cite[Tab. 3]{TZ},~\cite{LS},
$m_S=2$ for $S\simeq {\rm SL}(2,4)$, $m=2$~\cite{AT},
and $m_S=12$ for $S\simeq {\rm G}_2(4)$, $m=12$~\cite{AT}.
Since $m_S2^m\leq d \leq 2^{2m-1}$ only the last two cases
may occur.

For $S\simeq {\rm G}_2(4)$ degree $12$ is the only degree of a
non-trivial, irreducible,
projective
representation of degree  $\leq 64$.
By Proposition~\ref{Irreducible} there exists an irreducible
$G$-module $V'$ such that 
$\dim V'= 2^{12}-d =64\cdot 52$ and $52$ is the degree of
of an irreducible,
projective
representation of $S$, a contradiction.

Assume finally $m=2$. 
It follows from~\cite[Thm. 4]{Gri} that there \emph{exists}
a group $G=E\semi S$, $S\simeq {\rm SL}(2,4)$
and this group is unique up to isomorphism.
Using GAP or Magma 
one can compute characters
of $G$. For $H=Z(E)\times S$
there exist precisely two linear characters of $H$
with kernel $S$.
For any such character
$\lambda$ the induced character $\lambda^G$ is irreducible,
which
 rules out this possibility too.
\end{proof}




\subsection{Automorphism groups}
\label{Further results}
\label{Automorphism groups}
{\noindent \em Proof   of} Remark~\ref{AutGroup}.
For cases (i) and (ii) we refer to~\cite{Ho} and~\cite{Zhu}.
For the remaining two cases we have by the Main Theorem a finite subgroup
$G=E\semi S\leq \unA$, with $|E/(E\cap Z)|=p^{2m}$,  $Z=Z({\rm U}(V))$ and
 $S\simeq {\rm Sp}(2m,p)$.
The assertions follow in cases (iii) and (iv)
if $E/(E\cap Z)$ is normal in ${\rm Aut}\LL$, i.e. if ${\rm Aut}\LL$
has a regular, abelian normal subgroup.
Suppose ${\rm Aut}\LL$ has a non-abelian simple socle then,
by the classification of the $2$-transitive groups
(see~\cite{Ca}),  ${\rm Aut}\LL$ 
is at least triply transitive. In that case the application of 
Proposition~\ref{Irreducible} (to a point stabilizer) forces $\dim V= d=n-1$,
a contradiction. 
\qed



\begin{thebibliography}{99}




\bibitem{As}
M. Aschbacher:
 Finite group theory (2nd ed.).
Cambridge U. Press, Cambridge 2000. 




\bibitem{AT}  
J. H. Conway,
R. T. Curtis,
S. P. Norton,
R. A. Parker,
R. A. Wilson
and
J. G. Thackray:
 Atlas of finite groups.  Oxford U. Press, Eynsham 1985.
 
 
\bibitem{BRW} Beverley Bolt, T. G.  Room, G. E. Wall:
 On the Clifford Collineation, Transform and
Similarity Groups I, II, Journal of the Australian Mathematical Society 2 (1961), 60-79, 80-96.





\bibitem{Ca}
Peter J. Cameron: Primitive permutation groups and finite simple groups,
Bull. Lond. Math. Soc. 13(1981), 1-22. 

 





  \bibitem{Di} 
  Leonard Eugene Dickson:
 The groups of Steiner in problems of contact (second paper).
Trans. AMS 3  (1902),  377--382.
  

 
 
 
\bibitem{Ge} Paul G\'erardin:
 Weil representations associated to finite fields.
 J. Algebra 46 (1977) 54--101.


 
 

  \bibitem{Gri}  Robert L. Griess Jr.:
Automorphisms of extra special groups and nonvanishing
degree $2$ cohomology, 
Pac. J. Mat. 48(1973), 403-422.
 




\bibitem{Ho} 
Stuart G. Hoggar: 64 lines from a quaternionic polytope.
Geom. Ded. 69 (1998), 287--289.

\bibitem{HB} B. Huppert, N. Blackburn: 
\emph{Endliche Gruppen I}, Springer, 1967,
\emph{Finite groups II-III}, Springer, 1982.

\bibitem{Hu} B. Huppert: 
Singer-Zyklen in klassischen Gruppen, Mat. Z. 117(1970), 141-150.

\bibitem{IM}
Joseph W. Iverson and  Dustin G. Mixon: 
Doubly transitive lines II: Almost simple symmetries.  arXiv:1905.06859v1 [math.CO] 15 May 2019.


\bibitem{JP}  
Wayne  Jones and Brian Parshall:
On the 1-cohomology of finite groups of Lie type,
pp. 313--328 in:
Proc. Conf. Finite Groups (Utah 1975). 
Academic Press, New York  1976.

\bibitem{LS}  Vicente  Landazuri and Gary M. Seitz:  
On the minimal degrees of projective
representations of the finite Chevalley groups.
 J. Algebra  32 (1974), 418--443.    
 
\bibitem{Li}
Martin W. Liebeck: 
The affine permutation groups of rank three.
 Proc. LMS  54(1987),  477-516.





\bibitem{Th} J. Thompson: 
Nonsolvable finite groups all of whose local subgroups are solvable,
Bull. Amer. Math. Soc. 74(1968), 383-437.

\bibitem{TZ}
P. H. Tiep, A. E. Zaleskii: 
Some aspects of finite linear groups: a survey,
J. Math. Sci. (N.Y.) 100(2000), 1893-1914.

\bibitem{Ta1}
Donald E.  Taylor:
The geometry of the classical groups. Heldermann, 
Berlin 1992.  

\bibitem{Ta2}
D. E. Taylor:  
Two-graphs and doubly transitive groups.
JCT(A) 61  (1992), 113--122.  

\bibitem{Wi} D. Winter: 
The automorphism group of an extraspecial $p$-group, 
Rocky Mountain Jour. of Mathematics, 2(1972), 159-168.


\bibitem{Zhu}
Huangjun Zhu:  
Super-symmetric informationally complete measurements.
Ann. Physics 362 (2015) 311--326.


\end{thebibliography}
\end{document}